\documentclass[reqno,12pt]{amsart}
\usepackage{url}
\usepackage{amssymb,amsthm,amsmath,amsfonts,amstext,mathrsfs}
\usepackage{latexsym}
\usepackage[all]{xy}
\usepackage{xcolor}
\usepackage{enumerate}

\newtheorem{thm}{Theorem}[section]
\newtheorem{prop}[thm]{Proposition}

\newtheorem{coll}[thm]{Corollary}
\newtheorem{lem}[thm]{Lemma}
\newtheorem{ex}[thm]{Example}
\newtheorem{defn}[thm]{Definition}
\newtheorem{rmk}[thm]{Remark}

\newcommand{\sE}{{\mathcal E}}

\newcommand{\sM}{{\mathcal M}}

\newcommand{\sI}{{\mathcal I}}

\newcommand{\sL}{{\mathcal L}}
\newcommand{\sO}{{\mathcal O}}

\newcommand{\wu}{\widetilde{\mu}}

\newcommand{\sk}{\mathbf{k}}

\newcommand{\dd}{\widetilde{d}}

\input xy
\xyoption{all}

\title
{Slope inequalities and a Miyaoka-Yau type inequality}
\thanks{Gu is supported by the NSFC (No.11801391) and NSF of Jiangsu Province (No.BK20180832); Sun and Zhou are supported by the NSFC (No.11831013 and No.11501154);
 Zhou is also supported by NSF of
Zhejiang Province No. LQ16A010005}
\author{Yi Gu, Xiaotao Sun and Mingshuo Zhou}
\begin{document}


\begin{abstract}
For a minimal smooth projective surface $S$ of general type over a field of characteristic $p>0$, we prove that
$$K^2_S\le 32\chi(\sO_S).$$
Moreover, if $18\chi(\sO_S)<K^2_S\le 32\chi(\sO_S)$, Albanese morphism of $S$ must induces a genus two fiberation.
A classification of surfaces with $K^2_S=32\chi(\sO_S)$ is also given. The inequality also implies $\chi(\sO_S)>0$, which answers completely a question of Shepherd-Barron.
\end{abstract}

\maketitle

\tableofcontents
\section{Introduction}
Let $S$ be a smooth projective surface of general type over an algebraically closed field $\sk$.  When $\sk=\mathbb{C}$, we have the celebrated Miyaoka-Yau inequality (see \cite{Miyaoka, Yau}):
\begin{equation}\label{ieq: M-Y}
c^2_1(S)\le 3c_2(S).
\end{equation}
By the following Noether's formula (see \cite[Chap.~I, (5.5)]{BHPV}):
\begin{equation}\label{formula: Noether}
12\chi(\sO_S)=c_1^2(S)+c_2(S),
\end{equation}
the Miyaoka-Yau inequality (\ref{ieq: M-Y}) can also be  formulated as:
\begin{equation}\label{ieq: M-Y'}
c_1^2(S)\le 9\chi(\sO_S).
\end{equation}
The Miyaoka-Yau inequality (\ref{ieq: M-Y}) and its equivalent form (\ref{ieq: M-Y'}) play  a very important role in the study of complex algebraic surfaces (see {\it e.g.}, \cite{Beauville, Persson}).

Let us then turn to the positive characteristic cases. Assume that $\mathrm{char}.(\sk)=p>0$ now. Noether's formula (\ref{formula: Noether}) remains true (see \cite[\S~5]{Badescu}), but one should not expect the Miyaoka-Yau inequality (\ref{ieq: M-Y}). In fact, Raynaud's examples (see \cite{R} or \S~\ref{S: Raynaud's example} below) show that there exist minimal smooth projective surfaces $S$ of general type with $c_2(S)<0$, a contradiction to (\ref{ieq: M-Y}) as $c^2_1(S)=K_S^2>0$ when $S$ is minimal. As $c_2(S)$ can be negative, a natural question is that whether $\chi(\sO_S)$ can be negative or not.  In \cite{SB}, Shepherd-Barron shows that $\chi(\sO_S)> 0$ unless there is a suitable fibration $f: S\to C$ with singular general fibres of arithmetic genus $2\le g \le 4$ and $p\le 7$ (see \cite[Theorem~8]{SB} or Theorem~\ref{Thm: SB} below). However, the question if there exists such surface $S$ with $\chi(\sO_S)<0$  remains unsolved (see the Remark at page 268 of \cite{SB}). Shepherd-Barron also suggested that the most obvious place to look for such examples would be in the case where $(p,g)=(2,2)$. Later, it is proved by the first author in \cite{Gu} that $\chi(\sO_S)>0$ when $p\ge 3$. Our main result in this article is to prove a Miyaoka-Yau type inequality: $$K^2_S\le 32\chi(\sO_S)$$ for all smooth projective surfaces $S$ of general type, which in particularly implies $\chi(\sO_S)>0$ for any $p$ and answers Shepherd-Barron's question completely.

We observe that the above Miyaoka-Yau type inequality follows in fact a series of slope inequalities. Let $f:S\to C$ be a relatively minimal surface fibration  of genus $g\ge 2$ over $\sk$. The slope inequalities concern about the numerical relations between $K^2_{S/C}$ and $\chi_f:=\deg (f_*\omega_{S/C})$. When $\sk=\mathbb{C}$, we have the following so-called Xiao's slope inequality:
\begin{equation}\label{IEQ: Xiao}
K^2_{S/C}\ge \dfrac{4g-4}{g} \chi_f.
\end{equation}
This inequality is first  proved for any minimal fibration first by G. Xiao (see \cite{XG}) and for semi-stable fibrations independently by Cornalba and Harris (see \cite{C.H}). Some other proofs have also been given by others (see \cite{Moriwaki}).  In this paper, we first prove a partial generalization of Xiao's slope inequality in positive characteristics.

\begin{thm} \label{Thm: main-slope}
Let $f:S\to C$ be a relatively minimal fibration of genus $g\ge 2$ over an algebraically closed field $\sk$ of positive characteristic, assume any one of the following assumptions is true:
\begin{enumerate}[(a.)]
\item the generic fibre of $f$ is hyperelliptic;

\item the generic fibre of $f$ is smooth;

\item the genus $b:=g(C)\le 1$,
\end{enumerate} then Xiao's slope inequality
 $K^2_{S/C}\ge \dfrac{4g-4}{g}\chi_f$ holds.
\end{thm}
Note that in positive characteristics, the generic fibre of $f$ can be singular.  It should also be pointed out that in case the generic fibre of $f$ is singular, we do not have the semi-positivity of $K_{S/C}$ and $\chi_f$. As a result, both $K_{S/C}^2$ and $\chi_f$ can be negative (cf. \S~\ref{S: Raynaud's example} for an example). Under the assumption that $f$ has a smooth generic fibre, Xiao's slope inequality (\ref{IEQ: Xiao}) has already been proven in a previous paper \cite{SSZ} of H. Sun and the last two authors.

We then point out as one of the positive characteristic pathology, Xiao's slope inequality (\ref{IEQ: Xiao}) can fail in positive characteristics in general.
\begin{prop}[see \S~\ref{S: counterexample to Xiao}]\label{prop: co-Xiao}
For infinitely many integers $g\ge 3$, there exists relatively minimal surface fibration $f: S\to C$ of fibre genus $g$ over an algebraically closed field $\sk$ of positive characteristic such that $K_{S/C}^2<\dfrac{4g-4}{g}\chi_f$.
\end{prop}

For general fibrations or fibrations of small genus, we also give some different slope inequalities.
\begin{thm}\label{Thm: other slope}
Let $f: S\to C$ be a relatively minimal fibration of genus $g\ge 3$ over an algebraically closed field $\sk$ and $b:=g(C)$, then
\begin{enumerate}[(a.)]
\item  if $K_S$ is nef, then $K_S^2\ge \dfrac{2g-2}{g}\deg(f_* \omega_{S/\sk})$;

\item if $g=3$ and $f$ is non-hyperelliptic, then $K^2_{S/C}\ge 3\chi_f$;

\item if $g=4$ and $K_S$ is nef, then $7K_{S/C}^2\ge 15\chi_f-48(b-1)$;

\item if $g\ge 5$ and $K_S$ is nef, then $$K_{S/C}^2\ge \dfrac{2(g-1)(g-2)}{g^2-3g+1}\chi_f-\dfrac{4(g-1)(g^2-4g+2)}{g^2-3g+1}(b-1).$$
\end{enumerate}
\end{thm}
Over $\mathbb{C}$, the slope inequality for non-hyperelliptic genus-$3$ fibration in (b.) is proved by Horikawa (see \cite{Horikawa}) and Konno (see \cite{K1}). For non-hyperelliptic fibrations of $g=4,5$, Konno \cite{K} and Chen \cite{Chen} have also given some other slope inequalities.  In positive characteristics, Yuan and Zhang have given a slope inequality for general genus $g$ in \cite[Lem.~3.2]{Y.Z} involving terms of $b=g(C)$. Our slope inequality here in (c.) (d.) are different from theirs.

Now let us return to the Miyaoka-Yau type inequality. For a minimal surface $S$ of general type, when $c_2(S)\ge 0$, Noether's formula (\ref{formula: Noether}) already implies that $K^2_S\le 12\chi(\sO_S)$. So we assume $c_2(S)<0$, Shepherd-Barron has shown (see \cite[Thm.~6]{SB}) that the Albanese map of $S$ induces a fibration $f: S\to C$ of fibre arithmetic genus $g\ge 2$ and $b:=g(C)\ge 2$. By abuse of language, we call such $f:S\to C$ as the Albanese fibration of $S$.
As an application of Theorem~\ref{Thm: main-slope} and Theorem~\ref{Thm: other slope}, we have

\begin{thm}\label{Cor: slope 1}
Let $f: S\to C$ be the Albanese fibration of $S$. Then
$$K_S^2\le \left\{\begin{array}{cc}
\dfrac{(12g+8)(g-1)}{g^2-g-1}\chi(\sO_S), &\text{if generic fibre is hyperelliptic };\\
\\
18\chi(\sO_S), & \text{if \, $g=3$}; \\
\\
\dfrac{840}{47}\chi(\sO_S), & \text{if \, $g=4$}; \\
\\
\dfrac{12(g-1)(3g^2-4g-4)}{g(3g^2-12g+15)}\chi(\sO_S), & \text{if $g\ge 5$}.
\end{array}\right.$$
\end{thm}

By the above results and elementary computations, we have
\begin{thm} \label{Thm: main}
Let $S$ be a minimal smooth projective surface of general type. Then
$K_S^2\le 32\chi(\sO_S)$.
Moreover, when $18\chi(\sO_S)< K_S^2\le 32\chi(\sO_S)$, the Albanese fibration of $S$ must be a genus two fibration.
\end{thm}
Examples of surfaces $S$ with $K_S^2=32\chi(\sO_S)$ are given in \S~\ref{S: example of maximal slope}. Moreover, we give an example of $S$ whose Albanese fibration is of genus $3$ and $K_S^2=18\chi(\sO_S)$ in \S~\ref{S: counterexample to Xiao} to show that the inequality $18\chi(\sO_S)< K_S^2\le 32\chi(\sO_S)$ in the theorem is optimum.

This theorem then answers the question of Shepherd-Barron and leads to the following classification of surfaces with negative $\chi(\sO_S)$  after Liedtke \cite{Lie}. In addition, the above Miyaokao-Yau type inequality can also be used to study the canonical map of surfaces of general type as in \cite{Beauville}.

\begin{thm}[After \protect{\cite[Prop.~8.5]{Lie}}]
Let $S$ be a smooth projective surface over an algebraically closed field $\sk$ of characteristic $p>0$ with $\chi(\sO_S)<0$. Then:
\begin{enumerate}
\item either $S$ is birationally ruled over a curve of genus $1-\chi(\sO_S)$;

\item or $S$ is quasi-elliptic of Kodaira dimension one and $p=2$ or $3$.
\end{enumerate}
\end{thm}

This paper is organized as follows.

In \S~\ref{S: slope}, we first recall Xiao's approach for slope inequalities and then prove Theorem~\ref{Thm: main-slope} (=Theorem~\ref{Thm: Xiao's inequality}) and Theorem~\ref{Thm: other slope} (=Proposition~\ref{Prop:g=3} and Proposition~\ref{Prop: slope inequality small genus}).

In \S~\ref{S: M-Y},  we first prove Theorem~\ref{Cor: slope 1} (=Theorem~\ref{Thm: M_Y type}) by applying Theorem~\ref{Thm: main-slope} and Theorem~\ref{Thm: other slope} and then prove Theorem~\ref{Thm: main} (=Corollary~\ref{coro 1}).

Finally in \S~\ref{S: ex} we recall or give the following examples in positive characteristics:
\begin{itemize}
\item Raynaud's examples of minimal surfaces $S$ of general type with $c_2(S)<0$. In his examples, the Albanese fibration of $S$ is hyperelliptic with $g=\dfrac{p-1}{2}$ and meets the equality in Theorem~\ref{Thm: main} for hyperelliptic case. Moreover the Albanese fibration $f: S\to C$ meets Xiao's equality.

\item Examples of fibrations violating Xiao's slope inequality (hence proves Proposition~\ref{prop: co-Xiao}).

\item Examples of general type surfaces with $K_S^2=32\chi(\sO_S)$.
\end{itemize}

\begin{center}
{\bf Conventions}
\end{center}
\begin{itemize}

\item A surface fibration is a flat morphism $f: S\to C$ from a projective smooth surface to a smooth curve over an algebraically closed field such that $f_*\sO_S=\sO_C$. Note in particular, all geometric fibres of $f$ are connected. In positive characteristics, the general fibre of $f$ can be singular.

\item For a surface fibration $f: S\to C$, we denote by
\begin{itemize}
\item $K_S$ (resp. $K_C$): the canonical divisor of $S$ (resp. $K_C$);

\item $K_{S/C}:=K_S-f^*K_C$;

\item $\chi_f:=\deg(f_*\omega_{S/C})$;

\item $l(f):=\dim_\sk (R^1f_*\sO_S)_{\mathrm{tor}}$.
\end{itemize}
Let $g$ be the fibre genus of $f$ and $b:=g(C)$, then we have
\begin{align}
\label{Eq: K_f}K_{S/C}^2&=K_S^2-8(g-1)(b-1)\\
\label{Eq: chi_f}\chi_f&=\chi(\sO_S)-(g-1)(b-1)+l(f).
\end{align}

\item An integral curve over $\sk$ is called hyperelliptic if it admits a flat double cover to $\mathbb{P}_\sk^1$ and a surface fibration $f: S\to C$ is called hyperelliptic if a general fibre of $f$ is hyperelliptic.
\end{itemize}

\section{Slope inequalities of fibrations in positive characteristic}\label{S: slope}
We study some slope inequalities in positive characteristics in this section.  Our strategy is based on Xiao's approach on slope inequalities, and then, Xiao's slope inequality is proved for some special fibrations (see Theorem \ref{Thm: Xiao's inequality}). We also observe that Xiao's slope inequality can not hold for general case (see Remark \ref{rmk: counterexample to xiao}), and in turn, prove some other slope inequalities (see Proposition \ref{Prop:g=3} and Proposition \ref{Prop: slope inequality small genus}).
\subsection{Xiao's slope inequality in positive characteristics}\label{Sec: Xiao}
In the paper \cite{XG}, Xiao introduces an approach of slope inequality for surface fibrations  by studying the Harder-Narasimhan filtration of $f_*\omega_{S/C}$. For readers' convenience, we briefly recall the idea.

Let $C$ be a smooth projective curve over an algebraically closed field. For a vector bundle $E$ on $C$, let $$\mu(E):=\frac{\mathrm{deg}(E)}{\mathrm{rk}(E)}$$ where $\mathrm{rk}(E)$ and $\mathrm{deg}(E)$ denote the rank and degree of $E$ respectively. This vector bundle $E$ is called semi-stable if for any subbundle $E'\subseteq E$, one has $\mu(E')\leq \mu(E)$. One has the following well-known theorem.
\begin{thm}\label{Thm3.1} (Harder-Narasimhan filtration) For any vector bundle $E$ on $C$, there exists a unique filtration of subbundles  $$0:=E_0\subset E_1\subset \cdots \subset E_n=E$$
which is the so-called Harder-Narasimhan filtration, such that

(1) each subquotient bundle $E_i/E_{i-1}$ is semi-stable for $1\leq i\leq n$,

(2) $\mu_1>\cdots >\mu_n$, where $\mu_i:=\mu(E_i/E_{i-1})$ for $1\leq i\leq n$.
\end{thm}
We denote by  $\mu_{\min}(E)$ the last slope $\mu_n$ of $E$.

For a relatively minimal surface fibration $f: S\to C$ of fibre genus $g\ge 2$ over    $\mathbb{C}$, by using the Harder-Narasimhan filtration of $f_{\ast}\omega_{S/C:}$ $$0:=E_0\subset E_1\subset \cdots \subset E_n:=f_{\ast}\omega_{S/C}$$
Xiao constructs a sequence of effective divisors $$Z_1\geq Z_2\geq \cdots \geq Z_n\geq 0$$ such that
$N_i=K_{S/C}-Z_i-\mu_i F\,\,\,\,(1\le i\le n)$ are nef $\mathbb{Q}$-divisors. Here $F$ is a fibre of $f$ and $\mu_i:=\mu(E_i/E_{i-1})$.
Then he uses the following elementary lemma to get a lower bound of $K^2_{S/C}$.
\begin{lem} \label{Lem3.5} (\cite[Lem.~2]{XG})
Let $f: S\rightarrow C$ be a relatively minimal fibration, with a general fibre $F$. Let $D$ be a nef (resp. $f$-nef) divisor on $S$, and suppose that there are a sequence of effective divisors $$Z_1\geq Z_2\geq \cdots \geq Z_n\geq Z_{n+1}=0$$ (resp. such that $Z_n$ is vertical)
and a sequence of rational numbers $$\mu_1>\mu_2>\cdots >\mu_n,\quad \mu_{n+1}=0$$
such that for every $i$, $N_i:=D-Z_i-\mu_i F,i=1,...,n$ is a nef $\mathbb{Q}$-divisor. Then $$D^2\geq \sum_{i=1}^n(d_i+d_{i+1})(\mu_i-\mu_{i+1}),$$
where $d_i=N_i\cdot F$.
\end{lem}
Xiao's approach can not be applied directly to positive characteristics. A key point is the failure of the following lemma in positive characteristics.
\begin{lem}[\protect{\cite[Thm.~3.1]{MY}, or \cite[Lem.~3]{XG}}]\label{Lem: MY's lemma}
Over $\mathbb{C}$, for any vector bundle $E$ on $C$, the $\mathbb{Q}$-divisor $\mathcal{O}_{\mathbb{P}(E)}(1)-\mu_{\min}(E)\cdot \Gamma$ is nef on $\mathbb{P}(E)$.
\end{lem}
A key observation of \cite{SSZ} is that one can apply the following Lemma~\ref{Lem3.3} instead of Lemma~\ref{Lem: MY's lemma}   to generalize Xiao's approach to positive characteristics. We now assume $C$ is defined over an algebraically closed field $\sk$ with $\mathrm{char}.(\sk)=p>0$ until the end of this section.
\begin{lem}\label{Lem3.3}
If $E_i/E_{i-1}$ appearing in the Harder-Narasimhan filtration  of a vector bundle $E$ on $C$ are all strongly semi-stable (definition recalled below), then the $\mathbb{Q}$-divisor $\mathcal{O}_{\mathbb{P}(E)}(1)-\mu_{\min}(E)\cdot \Gamma$ is nef on $\mathbb{P}(E)$, here $\Gamma$ is a fibre of $\mathbb{P}(E)\to C$.
\end{lem}
Recall, let $F_C:C\rightarrow C$ be the (absolute) Frobenius morphism, a bundle $E$ on $C$ is called strongly semi-stable (resp., stable) if its pull back by $k$-th power $F_C^k$ is semi-stable (resp., stable) for any integer $k\geq 0$.

\begin{thm}\label{Thm: Langer} (\cite[Thm~3.1]{LA})  For any vector bundle $E$ on $C$, there exists an integer $k_0$ such that all  quotients $E_i/E_{i-1}$ ($1\leq i\leq n$) {appearing} in the Harder-Narasimhan filtration
$$0:=E_0\subset E_1\subset \cdots \subset E_n=F_C^{k\ast}E$$ are strongly semi-stable whenever $k\geq k_0$.
\end{thm}
\begin{rmk}\label{rmk: g(C<1)}
When $g(C)\le 1$,  semi-stable vector bundles are already strongly semi-stable (see \cite{sun}).  In particular, we can take $k=k_0=0$ in Theorem~\ref{Thm: Langer}.
\end{rmk}
Now suppose $f: S\to C$ is a relatively minimal surface fibration of fibre genus $g\ge 2$. Take $E=f_*\omega_{S/C}$ and fix a $k\ge  k_0$. Denote by
\begin{equation}\label{notion of HN}
0=E_0\subset \cdots \subset E_n=F_C^{k\ast}E=F_C^{k\ast}f_*\omega_{S/C}
\end{equation}
  the Harder-Narasimhan filtration, $r_i:=\mathrm{rk}(E_i), \mu_i:=\mu(E_i/E_{i-1})$ and call $\wu_i:=\dfrac{\mu_i}{p^k}$ the normalised slopes. Then we have
  \begin{equation}\label{formula of chi_f}
  \chi_f=\dfrac{1}{p^k}\deg F_C^{k\ast} E=\dfrac{1}{p^k}\sum\limits_{i=1}^n r_i(\mu_i-\mu_{i+1})=\sum\limits_{i=1}^n r_i(\wu_i-\wu_{i+1}).
  \end{equation}

Let us now recall the construction given in \cite{SSZ} of effective divisors $$Z_1\geq Z_2\geq \cdots \geq Z_n\geq 0$$ so that $N_i:=p^kK_{S/C}-Z_i-\mu_iF$ is nef. Considering the commutative diagram:
$$\xymatrix{ S\ar@/^20pt/[rrr]^{F^k_S}\ar[rr]^{F^k}
\ar[drr]^f && S'\ar[r]^{\alpha}\ar[d]^{f'}
& S\ar[d]^{f}\\
 && C\ar[r]^{F^k_C}& C} $$
there defines a natural morphism:
$$f^*E_i\hookrightarrow f^*F_C^{k*}E=F^{k*}_Sf^*f_*\mathcal{O}_S(K_{S/C})\to F^{k*}_S\omega_{S/C}=\mathcal{O}_S(p^kK_{S/C}),$$
for each $i$. Denote by $\mathcal{L}_i\subseteq F^{k*}_S\omega_{S/C}$ the image of this morphism and we can write $\mathcal{L}_i=\sI_{T_i}\cdot F^{k*}_S\omega_{S/C}(-Z_i)$ for a unique closed susbscheme $T_i$ of codimension $2$ and a unique effective divisor $Z_i$. It is clear by construction that $Z_1\ge Z_2\ge \cdots Z_n\ge 0$ and $Z_n$ is vertical. Let $U_i:=S\backslash{T_i}$, thus there is a morphism relative over $C$
$$\phi_i: U_i\to \mathbb{P}(E_i)$$
such that $\phi_i^*\mathcal{O}_{\mathbb{P}(E_i)}(1)=\mathcal{L}_i|_{U_i}$ by the construction of $\sL_i$.  Now we note that the $\mathbb{Q}$-divisor $c_1(\mathcal{L}_i)-\mu_i F$ is nef by Theorem~\ref{Thm: Langer} and Lemma~\ref{Lem3.3} as the complement of $U_i$ consists of finitely many points, here $F$ is a general fibre of $f$. In other words, we have the nefness of $$N_i=c_1(\mathcal{L}_i)-\mu_i F$$ by construction. Then Xiao's approach applies in positive characteristic and we have
\begin{thm}\label{Thm: Xiao's inequality}
Let $f:S\to C$ be a relatively minimal fibration of genus $g\ge 2$, assume any one of the following assumptions is true:
\begin{enumerate}[(a.)]
\item the generic fibre of $f$ is hyperelliptic;

\item the generic fibre of $f$ is smooth;

\item the genus $b:=g(C)\le 1$,
\end{enumerate} then Xiao's slope inequality
 $K^2_{S/C}\ge \dfrac{4g-4}{g}\chi_f$ holds.
\end{thm}

\begin{proof}
Let $d_i:=N_i\cdot F={\rm deg}(\sL_i|_{F}), d_{n+1}=p^k(2g-2)$ and $\dd_i:=\dfrac{d_i}{p^k}$, by Lemma~\ref{Lem3.5}, we have
\begin{align}\label{3.1.1}
p^{2k}K^2_{S/C} &\ge \sum^n_{i=1}(d_i+d_{i+1})(\mu_i-\mu_{i+1}) \ \  \text{or equivalently}
\end{align}
\begin{align}
 \label{3.1} K^2_{S/C}&\ge \sum^n_{i=1}(\dd_i+\dd_{i+1})(\wu_i-\wu_{i+1}).
\end{align}
  Then, if the Clifford type inequalities
\begin{equation}\label{3.2}
d_i\ge p^k(2r_i-2) \ \  \text{or equivalently} \ \ \dd_i\ge 2(r_i-1) \quad (1\le i\le n)
\end{equation}
hold, we will have by (\ref{3.1.1}) the following inequality (see \cite[pp.~695]{SSZ})
\begin{equation}\label{3.3}
K^2_{S/C}\ge 4\chi_f-\frac{2}{p^k}(\mu_1+\mu_n)=4\chi_f- 2(\wu_1+\wu_n) ,
\end{equation} which and inequality
\begin{equation*}
p^kK^2_{S/C}\ge (2g-2)(\mu_1+\mu_n) \ \  \text{or equivalently} \ \ K^2_{S/C}\ge (2g-2)(\wu_1+\wu_2)
\end{equation*}
 (obtained by applying Lemma~\ref{Lem3.5} to $D=p^k K_{S/C}, Z_1\ge Z_n\ge 0$ and $\mu_1\ge \mu_n$) imply
$$K^2_{S/C}\ge \frac{4g-4}{g}\chi_f$$
In conclusion,  our theorem follows if the Clifford type inequalities (\ref{3.2}) hold. It remains to estimate $d_i$ in order to prove (\ref{3.2}), consider the following commutative diagrams
$$\xymatrix{\ar@{-->}@/^20pt/[rr]^{\phi_i} S\ar[r]^{F^k}
\ar[dr]_f & S'\ar@{-->}[r]^{\phi_i'}\ar[d]^{f'}
& \mathbb{P}(E_i)\ar[dl]\\
 & C},\quad \xymatrix{\ar@/^20pt/[rrr]^{F^k_S} S\ar[rr]^{F^k}
\ar[drr]_f && S':=S\times_{C,F_C^k} C\ar[r]^{\ \ \ \ \ \ \alpha}\ar[d]^{f'}
& S\ar[d]^{f}\\
 && C\ar[r]^{F^k_C}& C}$$
where $\phi'_i: S'\dashrightarrow  \mathbb{P}(E_i)$ is defined by the image $\sL_i'\subseteq \alpha^*\omega_{S/C}$ of ${f'}^*(E_i)$:
$${f'}^*(E_i)\subset {f'}^*F_C^{k*}f_*\omega_{S/C}=\alpha^*f^*f_*\omega_{S/C}\to \alpha^*\omega_{S/C}.$$
Then ${\rm deg}(\phi_i')|{\rm deg}(\phi_{i-1}')$ and ${\rm deg}(\phi_n')={\rm deg}(\phi_{|\omega_{S'/C}|})$. Thus
$${\rm deg}(\phi_i')\ge 2 \quad (1\le i\le n)\,\,\text{when $f:S\to C$ is hyperellitic}.$$
Note that $\phi_i$ is well-defined on a general fiber $F$ and thus
$$d_i=N_i\cdot F=p^k{\rm deg}(\phi_i'){\rm deg}(\phi_i(F)), \quad {\rm deg}(\phi_i(F))\ge r_i-1.$$
or equivalently \begin{equation}\label{IEq for d_i}
\dd_i\ge \deg(\phi_i') (r_i-1).
\end{equation}
Now when $f:S\to C$ is hyperelliptic, we immediately have $$2=\deg(\phi'_n)\mid \deg(\phi'_i)$$ for all $i$ and hence the Clifford type inequality (\ref{3.2}) holds.

In non-hyperelliptic cases, if $\sL_i'$ is locally free long a general fibre $F'$ of $f':S'\to C$, then $\phi_i'$ is also defined along $F'$ and we have $d_i=p^k\deg(\sL_i'|_{F'}).$ Since $\sL'\subseteq\alpha^*\omega_{S/C}=\omega_{S'/C}$, $\sL'|_{F'}$ is a special line bundle and we have  $\deg(\sL_i'|_{F'})\ge 2r_i-2$ by Clifford's theorem (see \cite{Li}). Thus we have the desired Clifford type inequality
$$d_i=p^k\deg(\sL_i'|_{F'})\ge p^k(2r_i-2)\,\,(1\le i\le n).$$
Note that one sufficient condition for $\sL_i'$ to be locally free on $F'$ is the normality of the generic fibre $S\times_C \mathrm{Spec}(K(C))$ of $f'$. In fact, if the generic fibre is now normal, then $\sL'$ is automatically locally free on it. Hence $\sL'$ is locally free on a general fibre $F'$. The normality of the generic fibre $S\times_C \mathrm{Spec}(K(C))$ follows if (b.) the generic fibre of $f$ is smooth or (c.) $g(C)\le 1$. In case (b.) the generic fibre is moreover smooth and in case (c.) we can take $k=k_0=0$ by Remark~\ref{rmk: g(C<1)} and hence $f'=f$.

In conclusion, under either assumption of our theorem,  we have the Clifford type inequality (\ref{3.2}) and therefore $K_S^2\ge \dfrac{4g-4}{g}\deg f_*\omega_{S}$.
\end{proof}
\begin{rmk}\label{rmk: counterexample to xiao}
The Clifford type inequality (\ref{3.2}) can fail in general. As a result, Xiao's slope inequality does not hold in positive characteristics in general. We shall see such counterexamples in \S~\ref{S: counterexample to Xiao}.
\end{rmk}

\subsection{Other slope inequalities}
 We give some other slope inequalities in this subsection.
\begin{prop}[\cite{Horikawa, K1}]\label{Prop:g=3}
Let $f: S\to C$ be a non-hyperelliptic, relatively minimal surface fibration of genus $3$, then $K_{S/C}^2\ge 3\chi_f$.
\end{prop}
This result is well known over $\mathbb{C}$ and the  proof in \cite{K1} works in any characteristics.

\begin{prop}\label{Prop: slope inequality small genus}
Let $f: S\to C$ be a relatively minimal surface fibration of fibre genus $g$ such that $K_S$ is nef. Let $b:=g(C)$, then we have:
\begin{enumerate}
\item  $K_S^2\ge \dfrac{2g-2}{g}\deg(f_*\omega_{S/\sk})$;

\item if $g=4$, then  $7K_{S/C}^2\ge 15\chi_f-48(b-1)$;

\item if $g\ge 5$, then    $$ K_{S/C}^2\ge \dfrac{2(g-1)(g-2)}{g^2-3g+1}\chi_f-\dfrac{4(g-1)(g^2-4g+2)}{g^2-3g+1}(b-1).$$
\end{enumerate}
\end{prop}
\begin{proof}
We adopt the notations used in the previous section. Namely, let $E:=f_*\omega_{S/C}$ and  $$0=E_0\subset E_1\subset \cdots \subset E_n =F_C^{k\ast }E$$
be its Harder-Narasimhan filtration for some $k\ge k_0$ (see Theorem~\ref{Thm: Langer}). Take $\wu_i, r_i, Z_i, N_i$ and $\dd_i$ defined the same as in the previous section--in the paragraphs behind (\ref{notion of HN}). Then we see that $$N_1=p^kK_{S/C}-Z_1-\mu_1\cdot F\equiv p^kK_S-(\mu_1+2p^k(b-1))\cdot F$$ is nef. As $K_S$ is nef, we have
\begin{align*}
p^{2k}K_S^2&\ge p^k K_S\cdot N_1+ p^kK_S\cdot (p^kK_S-N_1)\\& \ge p^kK_S\cdot (\mu_1+2p^k(b-1))F \\
& \ge p^k(2g-2)(\mu_1+p^k(2b-2)).
\end{align*}
In other words,
\begin{equation}\label{Ieq mu_1}
K_S^2 \ge 4(g-1)(b-1)+(2g-2)\cdot \wu_1
\end{equation}
Note that by definition we have  $g\cdot \wu_1 \ge \chi_f=\deg f_*\omega_{S/\sk}-2g(b-1)$. Combining this inequality with the inequality (\ref{Ieq mu_1}), we obtain $$K_S^2\ge \dfrac{2g-2}{g}\deg(f_*\omega_{S/\sk}).$$
Since $K^2_{S}=K_{S/C}+8(g-1)(b-1)$, we can reformulate  (\ref{Ieq mu_1}) by:
  \begin{equation}\label{IEQ: mu_1}
   K_{S/C}^2 \ge (2g-2)\cdot \widetilde{\mu}_1-4(g-1)(b-1).
  \end{equation}
Let $v_i:=\widetilde{\mu}_1-\widetilde{\mu}_i$, so $0=v_1<v_2< \cdots <v_n$. By (\ref{formula of chi_f}) and (\ref{3.1}), we have:
\begin{align*}
K_{S/C}^2-4(g-1)\cdot \widetilde{\mu}_1&\ge - [\sum\limits_{i=1}^{n-1}(\widetilde{d}_i+\widetilde{d}_{i+1})(v_i-v_{i+1})+ 4(g-1)v_n] \\
\chi_f-g\cdot \widetilde{\mu}_1 &=-[\sum\limits_{i=1}^{n-1} r_i(v_{i+1}-v_i)+g v_n]
\end{align*}
With help of the following Lemma~\ref{lem: g=4,5}, immediately we have $$K_{S/C}^2-4(g-1)\cdot \widetilde{\mu}_1 \ge \left\{\begin{array}{cc}
5(\chi_f-g\cdot \widetilde{\mu}_1), & g=4\\
(2g-4)(\chi_f-g\cdot \widetilde{\mu}_1), & g\ge 5
\end{array} \right.$$
Combining  this inequality with (\ref{IEQ: mu_1}), after a simple calculation we then obtain the desired inequalities by eliminating $\widetilde{\mu}_1$.
\end{proof}

\begin{lem}\label{lem: g=4,5}
Let $ \Phi:=\sum\limits_{i=1}^{n-1}(\widetilde{d}_i+\widetilde{d}_{i+1})(v_i-v_{i+1})+ 4(g-1)v_n $ and $ \Psi:=\sum\limits_{i=1}^{n-1} r_i(v_{i+1}-v_i)+g v_n$. Then we have $\Phi\le \left\{ \begin{array}{cc}
\aligned 5\Psi, \,\,\,\,\,\, &g=4; \\
(2g-4)\Psi,  \,\,\, &g\ge 5.\endaligned
\end{array} \right.
 $
\end{lem}
\begin{proof}
Take $e_1:=\dd_1+\dd_2, e_i=\dd_{i+1}-\dd_{i-1}, i=2,...,n-1$ and $e_n=(2g-2-\widetilde{d}_{n-1})$. Then
$\Phi =\sum\limits_{i=1}^n e_i v_i$  and
$\Psi=\sum\limits_{i=1}^n (r_i-r_{i-1})v_i, r_0:=0$. Moreover, we have $\sum\limits_{i=1}^ne_i=4g-4$.

{\bf Case $ r_{n-1}<g-1$  or $n=1$}:
we have  when $g=4$,
\begin{align*}
5\Psi &\ge 10v_n+5v_{n-1}\ge e_n\cdot v_n+(15-e_n)\cdot v_{n-1} \\
      & \ge e_n\cdot v_n +(12-e_n)\cdot v_{n-1} \ge \Phi
\end{align*} since it is clear that $e_n\le 2g-2<10$ . And when $g\ge 5$,
\begin{align*}
(2g-4)\Psi &\ge (4g-8)v_n+(2g-4)v_{n-1}\\
      &\ge e_n\cdot v_n+(6g-12-e_n)\cdot v_{n-1} \\
      & \ge e_n\cdot v_n +(4g-4-e_n)\cdot v_{n-1} \ge \Phi
\end{align*} since it is clear that $e_n\le 2g-2< 4g-8$.

{\bf Case $r_{n-1}=g-1, r_{n-2}<g-2$ or $n=2$}: we have  when $g=4$,
\begin{align*}
5\Psi &\ge 5v_n+10v_{n-1}\ge e_n\cdot v_n+(15-e_n)\cdot v_{n-1} \\
      & \ge e_n\cdot v_n +(12-e_n)\cdot v_{n-1} \ge \Phi
\end{align*} since in this case $e_n=6-\dd_{n-1}\le 6-r_{n-1}+1=4<5$ by (\ref{IEq for d_i}). And when $g\ge 5$,
\begin{align*}
(2g-4)\Psi &\ge (2g-4)v_n+(4g-8)v_{n-1}\\
&\ge e_n\cdot v_n+(6g-12-e_n)\cdot v_{n-1} \\
      & \ge e_n\cdot v_n +(4g-4-e_n)\cdot v_{n-1} \ge \Phi
\end{align*} since in this case $e_n=2g-2-\dd_{n-1}\le g<2g-4$.

{\bf Case $r_{n-1}=g-1, r_{n-2}=g-2$}:
we have  when $g=4$,
\begin{align*}
5\Psi &\ge 5v_n+5v_{n-1}+5v_{n-2}\\
      &\ge e_n\cdot v_n+e_{n-1}\cdot v_{n-1}+(15-e_n-e_{n-1})\cdot v_{n-2} \\
      & \ge e_n\cdot v_n +e_{n-1}\cdot v_{n-1}+(12-e_n-e_{n-1})\cdot v_{n-2} \ge \Phi
\end{align*} since in this case $e_n\le 4$ and $e_{n-1}=6-\dd_{n-2}\le 6-r_{n-2}+1=5$. And when $g\ge 5$,
\begin{align*}
(2g-4)\Psi &\ge (2g-4)v_n+(2g-4)v_{n-1}+(2g-4)v_{n-2}\\
       &\ge e_n\cdot v_n+e_{n-1}\cdot v_{n-1}+(6g-12-e_n-e_{n-1})\cdot v_{n-2} \\
      &\ge e_n\cdot v_n+e_{n-1}\cdot v_{n-1}+(4g-4-e_n-e_{n-1})\cdot v_{n-2} \ge \Phi
\end{align*} since in this case $e_n\le g$ and $e_{n-1}=(2g-2)-\dd_{n-2}\le g+1\le 2g-4$.
\end{proof}
\section{Miyaoka-Yau type inequality in positive characteristics}\label{S: M-Y}
We then start to study the Miyaoka-Yau type inequality.  Suppose $S$ is a minimal surface of general type over an algebraically closed field $\sk$ with $\mathrm{char}.(\sk)=p>0$. If $c_2(S)>0$, we have an immediate Miyaoka-Yau type inequality $c_1^2(S)\le 12\chi(\sO_S)$ obtained from (\ref{formula: Noether}). Thus, it suffices to discuss for $S$ with $c_2(S)<0$.

We firstly recall a fundamental theorem on the structure of algebraic surfaces of general type with negative $c_2$ due to Shepherd-Barron.
\begin{thm}(Shepherd-Barron, see \cite[Theorem~8]{SB})\label{Thm: SB}
If $c_2(S)<0$, then the Albanese map of $S$ induces a fibration: $f: S\to C$ with
\begin{itemize}
\item $C$ is a nonsingular projective curve of genus $b:=g(C)\ge 2$ and $f_{\ast}{\sO_S}\cong \sO_C$;

\item the fibre (arithmetic) genus $g:=p_a(F)\ge 2$;

\item the geometric generic fibre is a singular rational curve with cusp singularity.
\end{itemize}
\end{thm}

By abuse of language, we call such $f:S\to C$ as the Albanese fibration of $S$. As an application of  Theorem~\ref{Thm: Xiao's inequality}, Proposition~\ref{Prop:g=3} and Proposition~\ref{Prop: slope inequality small genus}, we have
\begin{thm}\label{Thm: M_Y type}
Let $f:S\to C$ be the Albanese fibration of $S$. Then
$$K_S^2\le \left\{\begin{array}{cc}
\dfrac{(12g+8)(g-1)}{g^2-g-1}\chi(\sO_S), &\text{if generic fibre is hyperelliptic};\\

18\chi(\sO_S), & \text{if \, $g=3$}; \\

\dfrac{840}{47}\chi(\sO_S), & \text{if \, $g=4$}; \\

\dfrac{12(g-1)(3g^2-4g-4)}{g(3g^2-12g+15)}\chi(\sO_S), & \text{if $g\ge 5$}.
\end{array}\right.$$
\end{thm}

\begin{proof} We first recall the following numerical relations (\ref{Eq: K_f}) and (\ref{Eq: chi_f}):
$$\aligned K_{S/C}^2&=K_S^2-8(g-1)(b-1)\\
\chi_f&=\chi(\sO_S)-(g-1)(b-1)+l(f).\endaligned$$ Then if $f:S\to C$ is hyperelliptic, by Theorem~\ref{Thm: Xiao's inequality}, one has
\begin{align*}
K^2_S-8(g-1)(b-1)&\ge \dfrac{4g-4}{g}(\chi(\sO_S)-(g-1)(b-1)+l(f)) \\
&\ge \dfrac{4g-4}{g}(\frac{1}{12}(K^2_S+c_2(S))-(g-1)(b-1))\\
&=\frac{(g-1)}{3g}K_S^2+\dfrac{(g-1)}{3g}c_2(S)+\dfrac{4(g-1)^2(b-1)}{g}.
\end{align*}
By the inequality (see \cite[(3.3)]{Gu})
\begin{equation}\label{IEQ: for c_2}
c_2(S)\ge -4(b-1),
\end{equation} one has
$$ K^2_S\ge \frac{(12g+8)(g-1)}{2g+1}(b-1)\quad \text{or}\quad  \frac{4(b-1)}{K^2_S}\le \frac{2g+1}{(g-1)(3g+2)},$$
which implies the first inequality in Theorem~\ref{Thm: M_Y type}:
\begin{align*}
\frac{12\chi(\sO_S)}{K_S^2}&=1+\frac{c_2(S)}{K_S^2}\ge 1-\frac{4(b-1)}{K^2_S}\\
&\ge 1-\frac{2g+1}{(g-1)(3g+2)}=\frac{3(g^2-g-1)}{(3g+2)(g-1)}.
\end{align*}

Other inequalities follow the same computations. In fact, if we have a slope inequality $K^2_{S/C}\ge \varphi(g)\chi_f-\phi(g)(b-1)$ with $\varphi(g)<12$, we get
\begin{align*}
 & K_S^2-8(g-1)(b-1)\\
\ge& \varphi(g)(\dfrac{K_S^2+c_2(S)}{12}-(g-1)(b-1)+l(f))-\phi(g)(b-1) \\
                 \ge& \varphi(g)(\dfrac{K_S^2-4(b-1)}{12}-(g-1)(b-1))-\phi(g)(b-1).
\end{align*}
Thus
$$K_S^2\ge \dfrac{12(8-\varphi(g))(g-1)-4\varphi(g)-12\phi(g)}{12-\varphi(g)}(b-1).$$
So we have
\begin{align*}
\dfrac{12\chi(\sO_S)}{K_S^2} = 1+\dfrac{c_2(S)}{K_S^2}\ge 1-\dfrac{4(12-\varphi(g))}{ 12(8-\varphi(g))(g-1)-4\varphi(g)-12\phi(g)}\\
\end{align*}

Now by  Proposition~\ref{Prop:g=3} and Proposition~\ref{Prop: slope inequality small genus}, we can take
\begin{itemize}
\item $\varphi(3)=3, \phi(3)=0$, if $g=3$ and $f$ is not hyperelliptic;

\item $\varphi(4)=\dfrac{15}{7}, \phi(4)=\dfrac{48}{7}$, if $g=4$;

\item $\varphi(g)=\dfrac{2(g-1)(g-2)}{g^2-3g+1}, \phi(g)=\dfrac{4(g-1)(g^2-4g+2)}{g^2-3g+1}$, if $g\ge  5$.
\end{itemize}
Our theorem then follows from a simple calculation. Note that from the computation, when $f$ has genus $3$, then $K_S^2\le \dfrac{88}{5}\chi(\sO_S)<18\chi(\sO_S)$ if $f$ is hyperelliptic and $K_S^2\le 18\chi(\sO_S)$ if $f$ is non-hyperelliptic.
\end{proof}
\begin{rmk}\label{rmk to M_Y type}
1). From the proof, the inequalities in the theorem takes equality if and only if $c_2(S)=-4(b-1), l(f)=0$ and the associated slope inequality takes equality.

2). In \S~\ref{S: Raynaud's example} below, we will see that Raynaud's examples meets the equality for hyperelliptic fibrations in this theorem.

3). By Tate's genus change formula (cf. \cite{Tate} or \cite[\S~2.1]{Gu}), the genus $g$ is such that $(p-1)\mid 2g$.
\end{rmk}

\begin{coll}\label{coro 1}
Let $S$ be a minimal smooth projective surface of general type. Then
$K^2_S\le 32\chi(\sO_S)$.
Moreover, when $$18\chi(\sO_S)<K^2_S\le 32\chi(\sO_S),$$ the Albanese fibration of $S$ must be a genus two fibration.
\end{coll}
\begin{proof}
If $c_2(S)\geq 0$, Noether's formula implies $K_S^2\le 12\chi(\sO_S)$. It is enough to consider that $c_2(S)<0$, then we have Albanese fibration
$$f:S\to C$$
of genus $g\ge 2$. If $f:S\to C$ is hyperelliptic, by Theorem~\ref{Thm: M_Y type}, we have
$$K_S^2\le \frac{(12g+8)(g-1)}{g^2-g-1}\chi(\sO_S)$$
where $h(g)=\dfrac{(12g+8)(g-1)}{g^2-g-1}$ is an decreasing function of $g$ with $h(2)=32$.
Thus $K^2_S\le 32\chi(\sO_S)$. If $f:S\to C$ is non-hyperelliptic, we have $K_S^2\le 18\chi(\sO_S)$ for $g=3,4$ immediately from Theorem~\ref{Thm: M_Y type} and
$$  K_S^2\le \dfrac{12(g-1)(3g^2-4g-4)}{g(3g^2-12g+15)}\chi(\sO_S), \quad g\ge 5$$
where $n(g)=\dfrac{12(g-1)(3g^2-4g-4)}{g(3g^2-12g+15)}$ is also an decreasing function of $g$ with $n(5)=\dfrac{408}{25}<18$.
Thus $K^2_S\le 18\chi(\sO_S)$ when $g\ge 3$.  Altogether, we have
$$K^2_S\le 32\chi(\sO_S)$$
for all minimal smooth projective surfaces $S$ of general type and when $K^2_S>18\chi(\sO_S)$,
$f:S\to C$  must be a genus two fibration.
\end{proof}
In the next section, we shall construct the following examples:
\begin{enumerate}
\item examples of $S$ with $K_S^2=32\chi(\sO_S)$ (cf. \S~\ref{S: example of maximal slope}) ;

\item example of $S$ with $K_S^2=18\chi(\sO_S)$ but its Albanese fibration is of genus $3$ (cf. Proposition~\ref{exm: slope 18}).
\end{enumerate} So the bounds in Corollary~\ref{coro 1} is optimum.

To end this section, it is worthing to mention that Theorem \ref{Thm: M_Y type} implies Gu's conjecture for the ``hyperelliptic part'' (see Conjecture 1.4 of \cite{Gu}).
\begin{coll}\label{cor: hyperelliptic} Let $S$ be a minimal algebraic surface of general type in positive characteristic $p\geq 5$. Assume that $c_2(S)<0$ and the Albanese morphism $f:S\rightarrow C$ has generic hyperelliptic fibre, then we have
\begin{equation}\label{equ: Gu conjecture}\chi(\sO_S)\ge\dfrac{p^2-4p-1}{4(3p+1)(p-3)}K_S^2\end{equation}
and the equality holds exactly for Raynaud's example (see \S~\ref{S: Raynaud's example}).
\end{coll}
\begin{proof}
Since we always have $g\geq \dfrac{p-1}{2}$ by genus change formula (cf. \cite{Tate} or \cite[\S~2.1]{Gu}), then (\ref{equ: Gu conjecture}) is a direct consequence of Theorem \ref{Thm: M_Y type} by that $h(g)=\dfrac{(12g+8)(g-1)}{g^2-g-1}$ is an decrease function of $g$ with $$h(\dfrac{p-1}{2})=\dfrac{4(3p+1)(p-3)}{p^2-4p-1}.$$

If $S$ is one of Raynaud's example, the equality of (\ref{equ: Gu conjecture}) holds according to a direct computation (see \S~\ref{S: Raynaud's example}). Conversely, the equality of (\ref{equ: Gu conjecture}) holds only if $g=\dfrac{p-1}{2}$ by above statement, $c_2(S)=-4(b-1), l(f)=0$ and $K^2_{S/C}=\dfrac{4g-4}{g}\chi_f$ by Remark \ref{rmk to M_Y type}. The equality $c_2(S)=-4(b-1)$ holds only if all geometric fibre of $f$ is irreducible. Moreover, since when $g=\dfrac{p-1}{2}$, there is an integral horizontal divisor $\Delta$ contained in the non-smooth locus of $f$ such that $[\Delta: C]=p$. So there can be no multiple fibres in $f$. Namely each geometric fibre of $f$ is irreducible and reduced. And our result is a direct consequence of Lemma \ref{Lem: Characterization of Raynaud's surfaces} below.
\end{proof}

\section{Examples}\label{S: ex}
\subsection{Raynaud's examples}\label{S: Raynaud's example}
In the paper \cite{R}, Raynaud constructed a class of pairs $(S, \sL)$, where $S$ is a smooth projective algebraic surface in positive characteristic and $\sL$ is an ample line bundle on $S$ such that $H^1(S,\sL)\neq 0$.  These pairs then give counterexamples to Kodaira's vanishing theorem in positive characteristics. In fact, Raynaud's examples do not only violate  Kodaira's vanishing theorem, but also lead to many other pathologies in positive characteristic.

We now briefly recall his construction, one can also refer to \cite{R} or \cite[\S~4]{Gu}. Let us start with a smooth projective curve $C$ of genus $b:=g(C)\ge 2$ over an algebraically closed field $\mathbf{k}$ of characteristic $p>2$ equipped with a rational function $f\in K(C) \backslash K(C)^p$ such that $$\mathrm{div} (\mathrm{d} f)=pD$$ for some divisor $D$ on $C$. We have the following examples of $C$ known as a special case of the Artin-Schreier curves.
\begin{ex}[Artin-Schreier curves]
Let $C$ be the projective normal curve associated to the following plane equation:
$$y^p-y=\varphi(x), \varphi(x)\in \sk[x].$$
Then $\mathrm{div}(\mathrm{d}x)=(2b-2) \infty$, where $\infty \in C$ is the unique point at infinity. For a suitable choice of $\varphi(x)$  ({\it e.g., $\varphi(x)=x^{p+1}$}), the genus $b=g(C)$ can be such that $p\mid 2b-2>0$. And therefore $\mathrm{div}(\mathrm{d}x)=p\cdot D$ for $D=\dfrac{(2b-2)}{p}\infty$.
\end{ex}
Starting from $(C,f)$, Raynaud shows there is a suitable rank $2$ vector bundle $E$ on $C$ along with a  non-singular effective divisor $\Sigma$ on
$$\pi: P=\mathbb{P}(E)\to C$$ such that
\begin{itemize}
\item $\mathrm{det}(E)\simeq \sO_C(D)$;

\item $\Sigma\subsetneq P$ is a non-singular divisor consisting of two irreducible components $\Sigma_i, i=1,2$ such that
\begin{itemize}
\item  $\Sigma_1$ is a section of $\pi$ and $\Sigma_1\in \sO_P(1)$;

\item  $\pi: \Sigma_2\to C$ is inseparable (and hence isomorphic to the Frobenius morphism);

\item  $\Sigma_1\cap \Sigma_2=\emptyset$.
\end{itemize}
\end{itemize}
Moreover, all such configurations $(P(E),\Sigma)$ is coming from his construction by a suitable choice of $(C,f)$. And we actually have
\begin{itemize}
\item $\Sigma_2\in |\sO_P(p)\otimes \pi^*\omega^{-1}_{C/\sk}|.$
\end{itemize} In particular, the divisor $\Sigma$ is an even divisor on $P$, so we can construct a flat double cover $\sigma: S\to P$ with branch divisor $\Sigma$ by choosing any line bundle $\sM$ on $P$ with $\sM^2\simeq \sO_P(\Sigma)$. The obtained surface $S$ is smooth over $\sk$ since $\Sigma$ is so (see \cite[\S~2]{Gu}).
\begin{defn}[Raynaud's example]
Let $S$ be a smooth projective surface over $\mathbf{k}$. We call that $S$ is one of Raynaud's example if there is a flat double cover $\sigma:S\to P$   with branch divisor $\Sigma$.
$$
\xymatrix{S\ar[rr]^\sigma \ar[rd]_f&& P\ar[ld]^\pi\\
&C&}
$$
\end{defn} Note that by construction, the fibration $f: S\to C$ in Raynaud's construction is hyperelliptic.
Let $f:S\to C$ be one of Raynaud's examples associated to the triple $(C,P,\Sigma)$, then we have:
\begin{itemize}
\item  the fibre genus $g$ is such that $2g-2=p-3$;

\item $K_S^2=\dfrac{(3p^2-8p-3)(b-1)}{p}$;

\item $\chi(\sO_S)=\dfrac{(p^2-4p-1)(b-1)}{8p}$.
\end{itemize}
Thus
\begin{align}K_{S/C}^2&=K_S^2-8(g-1)(b-1)=- (p-1)(p-3)(b-1) \\
\chi_f&=\chi(\sO_S)-(g-1)(b-1)=-\dfrac{(p-1)^2(b-1)}{8p}.
\end{align}
In particular, we have equality in Xiao's inequality (see Theorem~\ref{Thm: main-slope}) $$K_{S/C}^2=\dfrac{4g-4}{g}\chi_f$$ while both sides of the equality are negative. Raynaud's example is such that $K_S^2=\dfrac{4(3p+1)(p-3)}{p^2-4p-1}\chi(\sO_S)$, which is the maximal possible slope for ``hyperelliptic part'' (see Corollary \ref{cor: hyperelliptic}).

We end this subsection by a characterization of Raynaud's example.
\begin{lem}\label{Lem: Characterization of Raynaud's surfaces}
Suppose $f: S\to C$ is a surface fibration. Then $S$ is one of Raynaud's example if and only if  $f$ satisfies:
\begin{enumerate}
\item [a)] every geometric fibre of $f$ is a singular rational curve of arithmetic genus $$g=\dfrac{p-1}{2};$$

\item [b)] every geometric fibre is hyperelliptic and integral.
\end{enumerate}
\end{lem}
\begin{proof}
The ``only if" part can be checked directly. Conversely,
let $\rho$ be the hyperelliptic involution, and $\sigma: S\to P':=S/\rho$ be the quotient map. Then condition b) implies that the canonical homomorphism $\pi:P'\to C$ has integral fibres. Note that $\pi:P\to C$ is birational to a ruled surface (recall that $K(C)$ is C1 by Tsen's Theorem) and $P'$ is normal with integral $\pi$-fibres, we see that $P'$ is exactly a smooth minimal ruled surface over $C$. Thus the quotient map $\sigma: S\to P'$ is a flat double cover with some branch divisor $\Sigma'\subsetneq  P'$, and $\Sigma'$ itself is smooth over $\mathbf{k}$ (see \cite[\S~2.2]{Gu}). on the other hand, it can be deduced from \cite[\S~2.1 \& 2.2]{Gu} that $\Sigma_{K(C)}:=\Sigma'\times_{\pi, C} K(C)$ is a divisor of $P'_{K(C)}:=P'\times_{\pi, C} K(C)\simeq \mathbb{P}_{K(C)}^1$ with \begin{itemize}
\item $\deg_{K(C)}\Sigma_{K(C)}=p+1$ (by $p_a=\dfrac{p-1}{2}$);
\item $\Sigma_{K(C)}$ contains a point inseparable over $K(C)$ (by the fact all fibres are singular).
\end{itemize}
By degree counting, it concludes that $(P',\Sigma')$ meets the configuration given by Raynaud. We are done.
\end{proof}

\subsection{Counterexample to Xiao's slope inequalities}\label{S: counterexample to Xiao}
Starting from the triple pair $(C,P=\mathbb{P}(\sE),\Sigma)$ constructed in the previous subsection, we can also take cyclic cover of $P$ branching at $\Sigma$ of higher degrees, which then give counterexamples to Xiao's slope inequality.

After an \'etale base change if necessary, we assume now $p+1\mid 2b-2$ and fix a line bundle $\sM$ on $C$ such that $\sM^{p+1}\simeq \omega_{C/\sk}$. Now as
$$\Sigma\in |\sO_P(p+1)\otimes \pi^*\omega^{-1}_{C/\sk}|=|(\sO_P(1)\otimes \pi^*\sM^{-1})^p|.$$ This data then gives a cyclic $(p+1)$ cover $\tau: S\to P$ branching at $\Sigma$. Since $\Sigma$ is a smooth divisor, $S$ is smooth over $\sk$. Denote by $f=\tau\circ \pi: S\to C$ the associated surface fibration. Then this fibration $f$ has the following properties:
\begin{itemize}
\item $S$ is a minimal surface of general type;

\item every closed fibre of $f$ is a singular rational curve of arithmetic genus  $\dfrac{p^2-p}{2}$;

\item $K^2_S=\tau^*(K_P+pc_1(\sO_P(1))+pf^* c_1(\sM))=(p+3)(p-2)(2b-2)$;

\item $\chi(\sO_S)=\sum\limits_{j=0}^p \chi(\sO(-j)\otimes \pi^*\sM^{-j})=\dfrac{p^2+p-8}{12}(2b-2)$;

\item $K_{S/C}^2=-(p^2-3p+2)(2b-2)<0$;

\item $\chi_f=-\dfrac{p^2-2p+1}{6}(2b-2)<0$.
\end{itemize}
So $$\dfrac{K_{X/C}^2}{\chi_f}=\dfrac{6(p^2-3p+2)}{p^2-2p+1}>\dfrac{4(p-2)(p+1)}{p(p-1)}=\dfrac{4g-4}{g},$$ but since $\chi_f<0$, this violates Xiao's slope inequality.

When $p=3$, we have $g=3$ and  $K_S^2=18\chi(\sO_S)$. Note that $f$ is clearly the Albanese fibration of $S$, and we have the next proposition.
\begin{prop}\label{exm: slope 18}
There is a surface $S$ of general type in characteristic $3$ with a genus $3$ Albanese fibration such that $K_S^2=18\chi(\sO_S)$.
\end{prop}

\subsection{Surfaces of general type with maximal slope}\label{S: example of maximal slope}
Let $S$ be a minimal surface of general type over an algebraically closed field $\sk$ with $\mathrm{char}.(\sk)=p$. We have $K_S^2\le 32 \chi(\sO_S)$ by Theorem~\ref{Thm: M_Y type}. When it comes $K_S^2= 32 \chi(\sO_S)$, we call the surface $S$ is \emph{of maximal slope}.
\subsubsection{Characterization of surfaces with maximal slope}
\begin{prop}\label{prop: criterion for maximal slope}
A general type surface $S$ is of maximal slope if and only if there is a fibration $f: S\to C$ of genus $2$ such that:
\begin{enumerate}
\item $b:=g(C)\ge 2$;

\item all fibres of $f$ is irreducible, singular and rational.
\end{enumerate}
\end{prop}
\begin{proof}
If $S$ is of maximal slope, its Albanese fibration $f: S\to C$ is a genus-$2$ fibration by Corollary~\ref{coro 1}. Moreover, from Remark~\ref{rmk to M_Y type}  when $S$ has maximal slope, one must have $c_2(S)=-4(b-1)$ which is equivalent to say that all fibres of $f$ are irreducible (see \cite[(3.3)]{Gu}).

Conversely, if $S$ admits such a fibration $f: S\to C$, then any fibre of $f$ can have only unibranch singularities and therefore $c_2(S)=-4(b-1)$ by Grothendieck-Ogg-Shafarevich formula. On the other hand, since all fibres of $f$ is irreducible and reduced (since genus-$2$ fibres have no multiplicity), we have $l(f)=0$ and the relative canonical map $$v: S\to P=\mathbb{P}(f_*\omega_{S/C})$$ is a morphism without base point.  In particular, we have $$\omega_{S/C}=v^*\sO(1).$$ Therefore $K_{S/C}^2=2 c^2_1(\sO(1))=2\deg (f_*\omega_{S/C})$. It then follows  from Remark~\ref{rmk to M_Y type} again that $K_S^2=32\chi(\sO_S)$.
\end{proof}

By Tate's genus change formula (see  \cite{Tate} or \cite[\S~2.1]{Gu}),  the fibration $f$ in Proposition~\ref{prop: criterion for maximal slope} can only occur possibly in characteristic $p=2,3$ or $5$.
\subsubsection{surfaces of maximal slope when $p=5$}
\begin{prop}
If $p=5$, a surface of general type is of maximal slope if and only if it is one of Raynaud's example.
\end{prop}
\begin{proof} The proof is due to Lemma \ref{Lem: Characterization of Raynaud's surfaces} and Proposition \ref{prop: criterion for maximal slope}.
\end{proof}

\subsubsection{surfaces of maximal slope when $p=2$}
We give another example of surface with maximal slope when $p=2$ as follows. Define $C$ to be the quintic plane curve given by homogeneous equation:
\begin{equation}
\label{equ: for C} Y^4Z+YZ^4=X^5
\end{equation}
over an algebraically closed field $\mathbf{k}$ of characteristic $p=2$. One can easily check that $C$ is a smooth curve of genus $b:=g(C)=6$. There are two affine subset $C_i$ ($i=0,1$) of $C$ as below:
\begin{itemize}
\item[$C_0$] $(Z=1): y^4+y=x^5, x=\dfrac{X}{Z}, y=\dfrac{Y}{Z}$, with $C\backslash C_0=\{(0,1,0)\}$;

\item[$C_1$] $(Y=1): z'^4+z'=x'^5, x'=\dfrac{X}{Y}=\dfrac{x}{y}, z'=\dfrac{Z}{Y}=\dfrac{1}{y}$,with $C\backslash C_1=\{(0,0,1)\}$;
\end{itemize}
\noindent For simplicity, we introduce the following notations:
\begin{itemize}
\item $\infty$ is the point $(0,1,0)$ which is the complement of $C_0$ in $C$;

\item $\Lambda:=\{(0,1,\lambda)|\lambda\in \mathbb{F}_{16}^*\}=\{(0,\tau,1)|\tau\in \mathbb{F}_{16}^*\}\subsetneq C$.

\item $C_1':=C_1\backslash \Lambda$;

\item $C_{10}=C_0\cap C_1'$.
\end{itemize}
%
%

\begin{itemize}
\item[Over $C_0$:] $S$ is defined as
\begin{equation}
Y_0^2=S_0T_0^5+xS_0^6;
\end{equation}
in the weighted projective space $\mathrm{Proj}(\sO_{C_0}[S_0^1,T_0^1,Y_0^3])$. Here the superscript on each element is its homogeneous degree.

\item[Over $C_1'$:] $S$ is defined as
\begin{equation}
Y_1^2=S_1T_1^5+\dfrac{x'}{1+z'^6}S_1^6;
\end{equation}
\end{itemize}
in the weighted projective space $\mathrm{Proj}(\sO_{C_1'}[S_1^1,T_1^1,Y_1^3])$.

The homegeneous translation relation is given by
$$
\left\{\begin{array}{cl}
S_1&=x'^3S_0 \\
T_1&=x'T_0 \\
Y_1&=x'^4Y_0+(1+z'^3)T_0^3
\end{array}
\right.
$$
and this construction makes sense because that $$x'^8(Y_0^2-S_0T_0^5+xS_0^6)=Y_1^2-(S_1T_1^5+\dfrac{x'}{1+z'^6}S_1^6)$$ and $x'$ is invertible on $C_{01}$.

One can check  easily that $S$ is a non-singular surface and the fibration $f:S\to C$ is as in Proposition~\ref{prop: criterion for maximal slope}. So it gives an example of surface with maximal slope in characteristic $2$. In this example, we actually have $\chi(\sO_S)=1$ and $K_S^2=32$. We would also like to mention that $S$ is given from $C\times \mathbb{P}^1$ by taking quotient of the foliation $D=s^6\dfrac{\partial}{\partial s}+\dfrac{\partial}{\partial x},$ where $s$ is the parameter of $\mathbb{P}^1$.

\subsubsection{No surfaces of maximal slope when $p=3$}
Finally we conclude that there is no surfaces of general type with maximal slope when $p=3$. Suppose we have such a surface $S$. Note that the relative canonical map gives a morphism: $\pi: S\to \mathbb{P}(f_*\omega_{S/C})$ since each fibre of $f$ is irreducible and reduced, and such morphism $\pi$ is necessarily a   flat double cover (see \cite[\S~2]{Gu}).
Let  $M\subsetneq \mathbb{P}(f_*\omega_{S/C})$ be the branch divisor of $\pi$, which satisfies:
\begin{itemize}
\item $M$ is a smooth, horizontal divisor and $[M:C]=6$;

\item each component of $M$ is inseparable over $C$;

\item for each point $c\in C$, its inverse image in $M$ has exactly two points as a set. In fact, otherwise there is some $c$ has one inverse image. Then the fibre of $f$ at $c$ by construction is a flat double cover of $\mathbb{P}^1_\mathbf{k}$ branching at a single point of multiplicity $6$. Such fibre is clearly not irreducible.
\end{itemize}
Then there are two possibilities:
\begin{enumerate}
\item [A)]  $M=M_1+M_2$ with $M_1\cdot M_2=0$, and the projections $M_i\to C$ ($i=1,2$) are both isomorphic to the Frobenius morphism;

\item [B)] $M$ is irreducible and the projection $u: M\to C$ factors as
$$
\xymatrix{
M\ar[rr]^{F_M}\ar[rd]^u&& M'\ar[ld]^v\\
&C&}
$$where $F_M$ is the frobenius morphism and $v$ is an \'etale double cover.
 \end{enumerate}
Indeed we only need to consider the case A), since replacing $C$ by the base change $v$ above which is an \'etale double cover, the case B) can be turned into A).

Finally we go to exclude case A). Let $\Sigma$ be the divisor class $\sO(1)$ of $\mathbb{P}(f_*\omega_{S/C})$, and $M_i\sim_{\mathrm{num}}3\Sigma+u_iF$ for $i=1,2$. Recall that $$\Sigma^2=\deg f_*\omega_{S/C}=\chi_f,$$ and we have
\begin{align}
2b-2&=(3\Sigma+u_iF)^2+(3\Sigma+u_iF)(-2\Sigma+(\chi_f+2b-2) F)\\
0&=(3\Sigma+u_1F)(3\Sigma+u_2F)
\end{align}
Thus $u_1=u_2$ and $b=1$, which is a contradiction.

\bigskip

\begin{center}
{\bf Acknowledgement}
\end{center}
The first named author would like to thank L. Zhang and T. Zhang for some helpful communications. We would like to thank
Christian Liedtke who suggested the application of $\chi(\sO_S)>0$ in classification of surfaces (Theorem 1.6) in an email to the second named author.

\noindent\address{Yi Gu: School of Mathematical Sciences, Soochow University, Jiangsu
215006, P. R. of China.
}\\
Email:  sudaguyi2017@suda.edu.cn\\
\noindent\address{Xiaotao Sun: Center of Applied Mathematics, School of Mathematics, Tianjin University, No.92 Weijin Road, Tianjin 300072, P. R. China}\\
Email: xiaotaosun@tju.edu.cn \\
\address{Mingshuo Zhou: Center of Applied Mathematics, School of Mathematics, Tianjin University, No.92 Weijin Road, Tianjin 300072, P. R. China}\\
Email: zhoumingshuo@amss.ac.cn
\end{document}